\documentclass[12pt]{article}
\usepackage{amsmath, amsthm, amssymb, verbatim, enumerate, url, stmaryrd} 
\usepackage[margin=1 in]{geometry}

\usepackage{graphicx}
\usepackage[all]{xy}

\newcommand{\C}{\mathbb C}

\newcommand{\gl}{\mathfrak{gl}}

\newcommand{\g}{\mathfrak g}

\newcommand{\h}{\mathfrak h}

\newcommand{\spec}{\text{Spec\,}}

\newcommand{\inv}{^{-1}}

\newcommand{\V}{\mathcal{V}}

\newcommand{\hW}{W}

\newcommand{\nks}{\widetilde{N_k}}

\newcommand{\Dn}{D_n}
\newcommand{\Ghat}{\widehat{G}}

\newcommand{\GH}{\langle G, \mathfrak{h} \rangle}

\newcommand{\MX}{{\mathcal{M}_{X}}}

\newcommand{\GLNk}{\mathcal{GL}_X(N_k(\widehat{A}_n))}
\newcommand{\GL}{\mathcal{GL}}
\newcommand{\onto}{\twoheadrightarrow}
\newcommand{\into}{\hookrightarrow}
\newcommand{\NVB}{\mathcal{NVB}}

\entrymodifiers={+!!<0pt,\fontdimen22\textfont2>}

\newtheorem{thm}{Theorem}[section]
\newtheorem{cor}[thm]{Corollary}
\newtheorem{lem}[thm]{Lemma}
\newtheorem{prop}[thm]{Proposition}
\newtheorem{claim}[thm]{Claim}

\theoremstyle{definition}
\newtheorem{rem}[thm]{Remark}

\theoremstyle{definition}
\newtheorem{example}[thm]{Example}

\theoremstyle{definition}
\newtheorem{defn}[thm]{Definition}

\title{An algebro-geometric construction of lower central series of associative algebras}

\author{David Jordan \& Hendrik Orem} 

\begin{document}

\maketitle
\begin{abstract}
The lower central series invariants $M_k$ of an associative algebra $A$ are the two-sided ideals generated by $k$-fold iterated commutators; the $M_k$ provide a filtration of $A$. We study the relationship between the geometry of $X = \spec A_{ab}$ and the associated graded components $N_k$ of this filtration.  We show that the $N_k$ form coherent sheaves on a certain nilpotent thickening of $X$, and that Zariski localization on $X$ coincides with noncommutative localization of $A$.  Under certain freeness assumptions on $A$, we give an alternative construction of $N_k$ purely in terms of the geometry of $X$ (and in particular, independent of $A$).  Applying a construction of Kapranov, we exhibit the $N_k$ as natural vector bundles on the category of smooth schemes.
\end{abstract}

\section{Introduction}
In \cite{K}, Kapranov proposes the following viewpoint on non-commutative algebraic geometry.  Given a non-commutative algebra $A$ with abelianization $A_{ab}$, the natural surjection $A\onto A_{ab}$ can be understood as an inclusion $X=Spec(A_{ab})\into Spec(A)$, where $Spec(A)$ is the (would-be) spectrum of the algebra $A$.  Kapranov then proposes to study so-called NC-complete algebras: the NC-completion of an algebra $A$ is the ``formal neighborhood" to the embedding $X\into Spec(A)$; thus considering NC-complete algebras $A$ amounts to studying formal thickenings of $X$ into a non-commutative scheme.

A basic invariant of a non-commutative algebra $A$ is its lower central series, the descending filtration $L_\bullet=L_\bullet(A)$ by Lie ideals, defined recursively as follows:
$$L_1:=A,\quad L_k:=[A,L_{k-1}].$$
In other words, $L_k$ is spanned by iterated commutators, $[a_1,[a_2,[\cdots [a_{k-1},a_k]\cdots],$ for $a_i\in A$.  The lower central series ideals are $M_k:=AL_kA=AL_k$, and the associated graded components are $N_k:=M_k/M_{k+1}$.  In particular, we have $N_1=A_{ab}=\mathcal{O}(X)$.  
In fact, the components $N_k$ depend only on the NC-completion of $A$, so they are really invariants of Kapranov's formal neighborhood of $X$ in $A$, rather than of $A$ itself.

For several years, evidence has been mounting that -- despite their origin in non-commutative algebra -- the associated graded components $N_k$, and their Lie analogs $B_k:=L_k/L_{k+1}$, should admit descriptions in terms of the \emph{commutative} algebraic geometry of $X=Spec(A/M_2)$.  The first such evidence was provided by Feigin and Shoikhet \cite{fs} for the free algebra, $A_n$, on $n$ generators.  They constructed an isomorphism $A_n/M_3\cong \Omega^{ev}(\mathbb{A}^n)$, identifying the first two components $N_1\oplus N_2$ with the space of even-degree differential forms on $X=\C^n$, with Fedosov quantized product (see Example \ref{star-ex}).  In \cite{DE} and \cite{EKM}, it is shown more generally that the components $B_k(A_n)$ and $N_k(A_n)$, respectively, are finite extensions of tensor field modules on $\mathbb{A}^n$, with respect to a natural action of the Lie algebra $W_n$ of polynomial vector fields.  The papers \cite{AJ}, \cite{BJ}, \cite{BB}, \cite{BEJKL}, \cite{BoJ}, \cite{Ker}, all apply geometric techniques echoing \cite{fs} to compute $B_k(A)$ and $N_k(A)$ in various special cases.

The present paper aims to explain the geometric nature of the lower central series filtration in two essential ways:  Firstly, we prove that the components $N_k(A)$ are Zariski local on $X$, and are in fact coherent sheaves on a certain canonical nilpotent thickening of $X$.  Secondly, under the assumption that $A$ is locally free -- in particular, for Kapranov's NC-manifolds, and for formally smooth (a.k.a quasi-free) algebras of Cuntz-Quillen \cite{CQ} -- we give an alternative construction of the lower central series in terms of the formal geometry of $X$.  This construction implies that $N_k(A)$ is not only local on $X$, but it is in fact formal:  $N_k(A)$ is completely determined by its value in a formal neighborhood of any point $x\in X$, which in turn has a universal answer depending only on the dimension of $X$.

An outline of the paper is as follows.  In Section 2.2, for an arbitrary finitely generated algebra $A$, we begin by introducing a commutative product on the quotient $A/M_3$ as follows:
\newtheorem*{prop:astar-assoc}{Proposition \ref{prop:astar-assoc}}
\begin{prop:astar-assoc} The star-product, $a\star b:=\frac12(ab + ba),$ is associative and commutative on $A_\star:=(A/M_3,\star)$, and makes each $N_k(A)$ into a finitely generated $A_\star$-module.
\end{prop:astar-assoc}

Kapranov has shown in \cite{K} that NC-complete algebras $A$ admit flat Ore localizations with respect to multiplicative subsets $S\subset A/M_2$ without zero divisors. Our first main result -- the content of Section 3 -- is that Ore localization of an algebra $A$ coincides with the ordinary, \emph{commutative} localization on the $A_\star$-module $N_k(A)$:

\newtheorem*{thm:ncloc}{Theorem \ref{thm:ncloc}}
\begin{thm:ncloc}
Fix a multiplicative subset $\bar{S}\subset (A/M_2)$ without zero divisors, and let $S:=\pi^{-1}(\bar{S})\subset A$.  
The map,
$$m: A_\star[S\inv] \otimes_{A_\star} N_k(A)\xrightarrow{\sim} N_k(A[S\inv]),$$
$$a\otimes n \mapsto an,$$
is an isomorphism.
\end{thm:ncloc}

The significance of this theorem is that it allows us to work with $N_k(A)$ as a coherent sheaf $\nks(A)$ on $X_\star$, applying tools from commutative algebraic geometry, such as localizations, completions, and Hilbert's syzygy theorem, to perform computations.  Let $\widehat{A}_n$ denote the degree completion of the free algebra $A_n$.  As applications of Theorem \ref{thm:ncloc} we have:
\newtheorem*{cor:ncloc}{Corollary \ref{cor:ncloc}}
\begin{cor:ncloc}  The completion $N_k(A)_{(\mathfrak{m}_x)}$ at any smooth point $x\in X^{sm}$ is a quotient of $N_k(\widehat{A}_n)$, where $n=dim X$.\end{cor:ncloc}
\newtheorem*{cor:onedim}{Corollary \ref{cor:onedim}}
\begin{cor:onedim}
Suppose that $X=Spec(A/M_2)$ is a zero or one dimensional scheme with finitely many non-reduced points.  Then each $N_k(A)$ is finite dimensional, for $k\geq 2$.
\end{cor:onedim}
\newtheorem*{cor:syzygy}{Corollary \ref{cor:syzygy}}
\begin{cor:syzygy}
Suppose $A$ is the quotient of a free algebra on generators $x_1,\ldots, x_n$ in degrees $d_1,\ldots d_n$, by a collection of homogeneous non-commutative polynomials, so that $A$ is graded by degree.  Then the Hilbert series for each $N_k(A)$ is a rational function, with poles at $d_i$th roots of unity.
\end{cor:syzygy}

In Section 4, we turn our attention to locally free algebras.

\begin{defn}
An algebra $A$ is called \emph{locally free of rank $n$} if, for every $x\in X$, there exists an isomorphism $A_{(\mathfrak{m}_x+M_2)} \cong \hat{A}_n$.
\end{defn}
\begin{example}  There are many examples of locally free algebras:
\begin{enumerate}[(i)]
\item The free algebra $A_n$ is locally free of rank $n$.
\item The coordinate algebra $\mathcal{O}(X)$ of any connected smooth affine curve $X$ is locally free of rank one.
\item The free product of locally free algebras of ranks $m$ and $n$ is locally free of rank $m + n$.
\item For any smooth complete intersection $X=\C[x_1,\ldots,x_{n+m}]/\langle f_1,\ldots, f_m\rangle$, the algebra $A=A_{n+m}/\langle\tilde{f}_1,\ldots, \tilde{f}_m\rangle$, where each $\tilde{f}_i \in f_i + M_2$, is locally free of rank $\dim X = n$.
\item If $A$ is locally free, and if $\bar{S}$ is any multiplicative subset of $A/M_2$ without zero divisors, then $A[S^{-1}]$ is locally free, where $S=\pi^{-1}(\bar{S})$.
\item Any NC-smooth algebra (see \cite{K}, Proposition 1.5.1), including Kapranov's NC-smooth thickenings of any finitely generated smooth commutative algebra (see \cite{K}, Proposition 1.6.1), is locally free.
\item Any formally smooth (a.k.a quasi-free) algebra is locally free, by the formal tubular neighborhood theorem (\cite{CQ}, Section 6, Theorem 2).
\end{enumerate}
\end{example}

\begin{rem} We have isomorphisms:
$$(A/M_2)_{(\mathfrak{m}_x)} \cong A_{(\mathfrak{m}_x+M_2)}/M_2 \cong \widehat{A}_n/M_2\cong \C[[x_1\ldots x_n]].$$ Thus if $A$ is locally free of rank $n$, then  $X$ is smooth of dimension $n$.
\end{rem}

Under these assumptions, we give a completely new construction of $N_k(A)$ in the language of formal geometry on $X$.  Given a module $M$ over the Lie algebra $W_n$, which satisfies certain finiteness conditions, we have the globalization $\GL_X(M)$, a sheaf on $X$ which ``spreads" $M$ around $X$, using the $W_n$-action to produce transition functions (see Section \ref{sec-nvb} for a review of formal geometry).  Our second main result is:
\newtheorem*{thm:formloc}{Theorem \ref{thm:formloc}}
\begin{thm:formloc}
We have $\nks(A) \cong \GLNk$ as vector bundles on $X$.  In particular, we have $N_k(A) \cong \Gamma(X, \GLNk)$, as $U(Vect(X))\ltimes O(X)$-modules.
\end{thm:formloc}

Combining Theorems \ref{thm:ncloc} and \ref{thm:formloc}, we exhibit $N_k$ as a \emph{natural vector bundle} - a functorial assignment to any smooth scheme $X$ of a vector bundle over $X$; remarkably this characterization is completely independent of the choice of locally free lift $A$.  It is well known that the category of natural vector bundles is equivalent to the category of finite dimensional representations of vector fields on a formal disc vanishing at the origin; this allows us to reduce the problem of describing the components $N_k$ for all locally free algebras of a given rank to a single, finite dimensional problem in representation theory; some example computations are given at the end of Section \ref{sec-nvb}.

\subsection*{Acknowledgments}
The existence of a description of the lower central series ideals via formal geometry was conjectured by Pavel Etingof.  We are grateful to him for sharing his conjecture with us, and for many helpful comments and suggestions as the paper progressed.  We are also grateful to David Ben-Zvi for many explanations concerning Fedosov quantization, formal geometry and the geometric interpretation of \cite{K}.

\section{Preliminaries}
For the remainder of the paper, $A$ denotes a finitely generated associative algebra, and $X:=Spec(A/M_2(A))$.
\subsection{NC-nilpotent algebras and localizations}
In this subsection, we compare the filtration of $A$ by lower central series ideals to Kapranov's NC-filtration \cite{K}.

\begin{defn} We consider the following filtrations on $A$:
\begin{enumerate}
\item The lower central series filtration $L_\bullet$ is defined recursively by $L_1=A, L_k:=[A,L_{k-1}]$.  
\item The associative lower central series filtration $M_\bullet$ is defined by $M_k:=AL_kA = AL_k$.
\item The NC-filtration $F_\bullet$ is defined by $F_k = \sum_{i_1 + \cdots + i_m= k-m} M_{i_1}\cdots M_{i_m}.$
\end{enumerate}
\end{defn}
\noindent We denote by $B_k:=L_k/L_{k+1}, N_k:=M_k/M_{k+1}$ the associated graded components of each filtration.  

\begin{rem}
  In Sections \ref{sec:NkCoh} and \ref{sec:NkFormal} we apply geometric techniques to the graded components of the filtration $M_k$. The same results in fact hold for the associated graded components of $F_k$, with easier proofs (compare, for example, Theorem \ref{thm:ncloc} with Theorem 2.1.6 of \cite{K}).
\end{rem}

\begin{defn} An algebra $A$ is \emph{Lie complete} (resp, \emph{NC-complete}) if it is complete in the topology induced by the filtration $M_\bullet$ (resp, $F_\bullet$). Similarly, $A$ is \emph{Lie nilpotent} (resp, \emph{NC-nilpotent} if $M_N(A)=0$ (resp, $F_N(A) = 0$) for some $N$.\end{defn}

Clearly, we have that $M_k\subset F_{k-1}$.  Conversely, (\cite{Jen}, Theorem 2) proves that for any finitely generated algebra, and any $k$, there exists $N$ such that $M_2^N\subset M_k$. In particular, $F_N\subset M_k$. Thus we have:

\begin{claim} 
\label{claim:top}
A finitely generated algebra $A$ is NC-complete if, and only if, it is Lie complete. Similarly, $A$ is NC-nilpotent if and only if it is Lie nilpotent.\end{claim}

Following \cite{K}, we will use the term NC-nilpotent for what is called Lie nilpotent in \cite{EKM}.

\begin{prop}
\label{prop:oreloc}
Suppose $A$ is NC-nilpotent.  Let $\overline{S} \subset (A/M_2)$ be a finitely generated multiplicative set without zero divisors, and define $S = \pi\inv(\overline{S})$, where $\pi: A \to A/M_2$ is the quotient map.Then: 
  \begin{enumerate}[(1)]
    \item $S$ satisfies the Ore condition.
    \item $A[S^{-1}]$ is finitely generated.
    \item $A[S^{-1}]$ is NC-nilpotent.
  \end{enumerate}
\end{prop}
\begin{proof}
Claim (1) is precisely Proposition 2.1.5 in \cite{K}.  By Claim \ref{claim:top}, $F_\bullet$ and $M_\bullet$ induce the same completion, and hence $A$ is nilpotent for one filtration if and only if it is nilpotent for the other.  Claim (3) thus follows from Theorem 2.1.6 in \cite{K}.

To establish finite generation, we show that $(S+M_2)\inv$ is generated by $A$ and $S^{-1}$.  For $s\in S$ and $m\in M_2$, we express $(s+m)\inv$ as a geometric series,
  \begin{align*}
    (s + m)\inv &= s\inv(1 + s\inv m)\inv \\
    &= s\inv \left( \sum_{i \geq 0} (-s\inv m)^i \right),
  \end{align*}
which is seen to be a finite sum using the NC-nilpotence of $A$, and identity (2.1.7.2) in \cite{K}:
  \[ ms\inv = s\inv m + s^{-2}[s, m] + s^{-3}[s, [s, m]] + \cdots \]
\end{proof}

\subsection{The algebra $A_\star$}
The algebra $A/M_3$ carries a commutative multiplication, in addition to the non-commutative multiplication induced from $A$.  We have:

\label{sec:Astar}
\begin{prop}\label{prop:astar-assoc}
The star-product, $a\star b:=\frac12(ab + ba),$ is associative on $A_\star:=(A/M_3,\star)$, and the action $a\star n := \frac12 (a n + n a)$, for $a\in A_\star,n\in N_k(A)$ makes each $N_k(A)$ into an $A_\star$-module.
\end{prop}

The proof of Proposition \ref{prop:astar-assoc} depends on the following two lemmas:

\begin{lem}[Corollary 6.4 in \cite{BJ}]
If $i$ is odd, then $M_iM_j \subset M_{i+j-1}$.
\label{lem:M2}
\end{lem}

\begin{lem}\label{lem:jump2}	We have the containment $[A,[A,M_k]] \subset M_{k+1}$.\end{lem}
\begin{proof}
We let $a,b,c \in A, d\in L_k$, and apply Leibniz identity:
$$ [a,[b,cd]] = [a,[b,c]d] + [a,c[b,d]] = [a,[b,c]]d + [b,c][a,d] + [a,c[b,d]].$$
Applying Lemma \ref{lem:M2}, we have that all three terms on the right are in $M_{k+1}$.  
\end{proof}

\begin{proof}[Proof of Proposition \ref{prop:astar-assoc}]
To see that $\star$ is associative, we compute:
$$ (a\star b)\star c - a\star (b\star c) = \frac14\left((ab+ba)c +c(ab+ba) - a(bc+cb) - (bc+cb)a\right) = \frac14[b,[a,c]],$$
and thus associativity holds, modulo $M_3$.

The containment $M_3M_k\subset M_{k+2}$ of Lemma \ref{lem:M2} shows the map of vector spaces $A/M_3\otimes N_k\to N_k$ is well-defined.  Associativity of the action follows from Lemma \ref{lem:jump2}; for $n\in M_k$, we have:
$$(a \star b) \star n - a \star (b \star n) =  \frac{1}{4}[b,[a,n]] = 0 \mod M_{k+1}.$$
\end{proof}

\begin{defn}
We denote by $A_\star$ the algebra ($A/M_3, \star$), and we let $X_\star:=Spec A_\star$.
\end{defn}

\begin{example}\label{star-ex}
When $A = A_n$, the Feigin-Shoikhet isomorphism \cite{fs} identifies $A/M_3$ with the algebra $\Omega^{ev}(\C^n)$ of even degree differential forms on $\C^n$, equipped with the Fedosov quantized product,
$$\omega \nu := \omega\wedge\nu + d\omega\wedge d\nu.$$
In this case $A_\star$ is the commutative algebra of even-degree differential forms, with ordinary product.  \end{example}

\section{Lower central series components as coherent sheaves on $X_\star$}
\label{sec:NkCoh}
Throughout this section we assume that $A$ is an NC-nilpotent algebra, with $M_N(A)=0$.  In the previous section, we endowed each $N_k(A)$ with the structure of a module for the commutative algebra $A_\star$, and hence obtained a quasi-coherent sheaf $\widetilde{N_k}$ on $X_\star$.  On the other hand, $N_k$ also determines a pre-sheaf on $X_\star$ by the assignment on affine opens $U_{\bar{S}}\subset X_\star \mapsto N_k(A[S^{-1}])$.  The purpose of this section is to show that the two pre-sheaves obtained this way coincide; this is the content of:
\begin{thm}
Fix a multiplicative subset $\bar{S}\subset (A/M_2)$ without zero divisors, and let $S:=\pi^{-1}(\bar{S})\subset A$.  
The map,
$$m: A_\star[S\inv] \otimes_{A_\star} N_k(A)\xrightarrow{\sim} N_k(A[S\inv]),$$
$$a\otimes n \mapsto an,$$
  \label{thm:ncloc}
is an isomorphism.
\end{thm}

\begin{rem}
It follows from \cite{EKM}, Theorem 2.1, that $N_k(A)$ is finitely generated over $A/M_3$ in the case that $A$ is a free algebra of finite rank, and thus whenever $A$ is finitely generated.  Thus each $N_k(A)$ is a coherent sheaf over $X_\star$.
\end{rem}
\begin{rem}
We note that there is no loss of generality in assuming $A$ is NC-nilpotent: we have $N_k(A) \cong N_k(A/M_{k'})$ for $k' > k$; this allows us to apply Theorem \ref{thm:ncloc} to an arbitrary finitely generated algebra by first quotienting by sufficiently large $M_{k'}$.
\end{rem}

The remainder of this section is devoted to proving Theorem \ref{thm:ncloc}.  We will make repeated use of the following computation:

\begin{lem}\label{lem:inv-forms}Let $f \in A$ invertible.  Then we have:
  $[f\inv, g] = -f^{-2}[f, g] \mod M_3(A)$.
\end{lem}
\begin{proof}
  We compute:
  \begin{align*}
    0 &= f^{-1}[f\inv f, g] \\
    &= f^{-1}(f\inv[f, g] + [f\inv, g]f) \\
    &= f^{-1}(f\inv[f, g] + f[f\inv, g] + [ [f\inv, g], f]) \\
    &= f^{-2}[f, g] + [f\inv, g] + f\inv[ [f\inv, g], f].
  \end{align*}
\end{proof}
\begin{rem}
In the Feigin-Shoikhet isomorphism $A/M_3\cong \Omega^{ev}(\C^n)$, a commutator $[f,g]$ is sent to the two-form $2df\wedge dg$.  We note that Lemma \ref{lem:inv-forms} is compatible with the Feigin-Shoikhet isomorphism, via the identity $d(\frac{1}{f})=\frac{-1}{f^2}df$.
\end{rem}

For elements $a_1, \ldots, a_l \in A$, we denote by $M_k(a_1, \ldots, a_l)$ the image of $M_k(A_l)$ under the homomorphism $A_l \to A$, given by $x_i \mapsto a_i$.

\begin{lem}
  For $m \in M_k$, the expression 
  \begin{equation}
  [b, d][a, m] + [a, d][b, m]
  \end{equation}
  lies in $M_{k+2}$.
  \label{lem:FSprod}
\end{lem}
\begin{proof}
The expression maps to zero under the Feigin-Shoikhet map, so it lies in $M_3(a, b, d, m)$.  The space $M_3(a, b, d, m)$ is spanned by the elements:
$$m[a,[b,d]], a[b,[d,m]], [a,[b,dm]], [ab,[d,m]], [a,[bd,m]],$$ together with their permutations in the symbols $a, b, d$.  We have $m[a, [b, d]]\in M_kM_3 \subset M_{k+2}$, by Lemma \ref{lem:M2} while the remaining elements lie in $M_{k+2}$, by Lemma \ref{lem:jump2}.  
\end{proof}

\begin{thm}\label{thm:invcomm}
  For $k$ odd, we have: 
$$    [ A, [A, S\inv M_k(A)]] \subset S\inv M_{k+2}(A) + M_{k+3}(A[S^{-1}]).$$
\end{thm}
\begin{proof}
Let $T= S^{-1}M_{k+2}(A) + M_{k+3}(A[S^{-1}]).$  We need to show that, for $a,b\in A, d\in S$, and $m\in M_k(A)$, we have $[a,[b,d\inv m]]=0\mod T.$  The proof consists of moving $d^{-1}$ out of the nested commutator, modulo terms in $M_{k+3}(A[S^{-1}])$.  We compute: 
\begin{align}
   [a, [b, d\inv m]] &= [a, [b, d\inv] m] + [a, d\inv[b, m]]\label{first-step}\\
   &= -[a, d^{-2}[b, d]m] + [a,d^{-1}[d^{-1},[b,d]]m] + [a, d\inv][b, m] + d^{-1}[a,[b,m]]\label{second-step}\\
   &= -[a, d^{-2}[b, d]m] + [a, d\inv][b, m] + d^{-1}[a,[b,m]] \mod L_{k+3}(A[S^{-1}])\label{third-step}
\end{align}
In passing from equation \eqref{first-step} to \eqref{second-step}, we have manipulated the first term as in the proof of Lemma \ref{lem:inv-forms}, and applied Jacobi identity to the second term.  Because $k$ is odd, Theorem 1.3 of \cite{BJ} implies that the second term in \eqref{second-step} lies in $[A,M_3M_k]\subset [A,M_{k+2}] \subset L_{k+3}(A[S^{-1}])$, so it is zero modulo $T$.  The third summand in equation \eqref{third-step} lies in $T$, so may be ignored.  Expanding the first summand of equation \eqref{third-step} with the Leibniz rule, and applying Lemma \ref{lem:inv-forms} to the second, we obtain:
\begin{align}
[a,[b,d\inv m]] &= -d^{-2}[a, [b, d]m] - [a, d^{-2}][b, d]m - d^{-2}[a, d][b, m]\mod T.\label{fourth-step}
\end{align}
Applying Lemma \ref{lem:inv-forms} to the second RHS term of \eqref{fourth-step} gives:
$$2d^{-3}[a, d][b, d]m \in S^{-1}M_3(A)M_k(A)\subset S^{-1}M_{k+2}(A)\subset T,$$
as in the proof of Lemma \ref{lem:FSprod}.  Thus we have: 
\begin{align*}
  [a, [b, d\inv m]] &= -d^{-2}\left([a, [b, d]m] + [a, d][b, m]\right) \mod T\\
  &= -d^{-2}\left([b, d][a, m]  + [a, d][b, m]\right) -  d^{-2}[a, [b, d]]m \mod T,
\end{align*}
by the Jacobi identity.  By Lemma \ref{lem:FSprod} the first expression lies in $S^{-1}M_{k+2}(A)\subset T$; by Lemma \ref{lem:M2}, the second expression does as well.  
\end{proof}

\begin{cor}\label{cor:tech-result}
We have:
\begin{enumerate}
\item $M_k(A[S^{-1}]) = A[S^{-1}]\otimes_A M_k(A) + M_{k+1}(A[S^{-1}])$.
\item $M_k(A[S^{-1}]) = A[S^{-1}] \otimes_A M_k(A)$.
\end{enumerate}
\end{cor}
\begin{proof}
Assertion (2) follows by repeatedly applying (1), and using that $A[S\inv]$ is NC-nilpotent. For (1), we consider an arbitrary generator of $M_k(A[S^{-1}])$,
$$v=a[l_1,[l_2,[\ldots [l_{k-1},l_k]\cdots ],$$
 where the $l_i$ are some monomials in the generators of $A[S^{-1}]$. If each $l_i$ lies in $A$, there is nothing to prove; thus, we suppose that $l_i=\tilde{l_i}f^{-1}$, for some $\tilde{l_i}\in A$, $f\in S$, and $i\in I$.  By repeated use of Jacobi identity, we may assume that $i=k$.

By repeatedly applying Theorem \ref{thm:invcomm}, we may pull the $f^{-1}$ out of the nested commutator two slots at a time, until we are left with expressions of the form:

$$af^{-1}[l'_1,[l'_2,\ldots[l'_{k-1},l'_k]\cdots], \textrm{ if $k$ is odd, or}$$
$$a[l'_1,f^{-1}[l'_2,\ldots[l'_{k-1},l'_k]\cdots], \textrm{ if $k$ is even,}$$
where we have ignored any additional terms produced which lie in $M_{k+1}(A[S\inv])$.
In the first case, there is one less inversion in the iterated commutator expression, so by induction $v$ lies in $A[S^{-1}]\otimes_A M_k(A) + M_{k+1}(A[S^{-1}])$.  In the second case, we compute, using the Jacobi identity:
{\small\begin{align*}
&a[l'_1,f^{-1}[l'_2,\ldots[l'_{k-1},l'_k]\cdots]\\ &= af^{-1}[l'_1,[l'_2,\ldots[l'_{k-1},l_k]\cdots] + a[l_1',f^{-1}][l'_2,\ldots[l'_{k-1},l'_k]\cdots]\\
&= af^{-1}[l'_1,[l'_2,\ldots[l'_{k-1},l_k]\cdots] -af^{-2}[l_1',f][l'_2,\ldots[l'_{k-1},l'_k]\cdots] \mod M_{k+1}(A[S\inv]).
\end{align*}}
The second summand above lies in $M_2M_{k-1}\subset M_k$, and both summands feature one less inversion in the nested commutator, and so by induction lie in $A[S^{-1}]\otimes_A M_k(A) + M_{k+1}(A[S^{-1}])$. \end{proof}

\begin{proof}[Proof of Theorem \ref{thm:ncloc}]
  Having proved Corollary \ref{cor:tech-result}, the proof echoes Theorem 2.1.6 of \cite{K}.  Namely, we have an exact sequence:
\begin{align}
0 \to M_{k+1} \to M_k \to N_k \to 0.
\end{align}
By Proposition \ref{prop:oreloc}, $A[S^{-1}]$ is a flat $A$-module; applying $A[S\inv]\otimes_A -$ we obtain: 
\begin{align}
0 \to A[S\inv] \otimes_A M_{k+1} \to A[S\inv] \otimes_A M_k \to A[S\inv] \otimes_A N_k \to 0.
\end{align}
Combining this with Corollary \ref{cor:tech-result}, we obtain 
\begin{align}
A[S\inv] \otimes_A N_k(A) &\cong \frac{A[S\inv] \otimes_A M_k}{A[S\inv] \otimes_A M_{k+1}} \cong \frac{M_k(A[S\inv])}{M_{k+1}(A[S\inv])} \cong N_k(A[S\inv]).
\end{align}
\end{proof}

\begin{cor}\label{cor:ncloc} The completion $N_k(A)_{(\mathfrak{m}_x)}$ at any smooth point $x\in X^{sm}$ is a quotient of $N_k(\widehat{A}_n)$, where $n=dim X$.\end{cor}

\begin{cor}\label{cor:onedim}
Suppose that $X=Spec(A/M_2)$ is a zero or one dimensional scheme with finitely many non-reduced points.  Then each $N_k(A)$ is a finite dimensional vector space, for $k\geq 2$.
\end{cor}
\begin{proof}
By Theorem \ref{thm:ncloc}, each $N_k(A)$ forms a coherent sheaf on $X_\star$.  Over any smooth, reduced point of $X$, the fiber is isomorphic to $N_k(A_n)=0$, for $n = 0, 1$; hence the sheaf is supported only over the singular and non-reduced points of $X$, where it is of finite rank.
\end{proof}

\begin{rem}
The assumption that there are finitely many non-reduced points is necessary. For example, let $A = k\langle x, y \rangle /(y^2)$. Then the set $\{[x^l,y],l\geq 1\}$ is linearly independent in $N_2$.  See \cite{CF} for an elaboration of this example, and related examples.
\end{rem}

\begin{cor}\label{cor:syzygy}
Suppose $A$ is the quotient of a free algebra on generators $x_1,\ldots, x_n$ in degrees $d_1,\ldots d_n$, by a collection of homogeneous non-commutative polynomials, so that $A$ is graded by degree.  Then the Hilbert series for each $N_k(A)$ is a rational function, with poles at $d_i$th roots of unity.\end{cor}
\begin{proof}
This follows immediately from Hilbert's syzygy theorem for graded modules over polynomial rings.
\end{proof}

\section{Formal geometry and locally free algebras}
\label{sec:NkFormal}
Throughout this section, we assume $A$ is a finitely generated, locally free algebra of rank $n$. We begin with a review of the basic constructions in formal geometry, and then give an alternative construction of $N_k(A)$ in terms of the formal geometry on $X$. For background, we refer to \cite{BK}, Section 3, whose conventions we follow.
\subsection{Harish-Chandra torsors}
\label{sec:HCT-review}
Let $G$ be a pro-algebraic group, $G = \varprojlim G^k$.  Let $\g=Lie(G)$ denote the Lie algebra of $G$, and let $\rho:\mathcal{M}\to X$ be a $G$ torsor, that is, a scheme $\mathcal{M}$ over $X$, equipped with a $G$ action commuting with $\rho$, and inducing an isomorphism $G\times \mathcal{M}\cong \mathcal{M}\times_X\mathcal{M}$.

\begin{defn}
We say a $G$-module $N$ is a \emph{pro-finite} if we have $N = \varprojlim N^k$ of finite dimensional representations $N^k$ of the algebraic groups $G^k$.
\end{defn}
Throughout this section, we take all $G$-modules to be pro-finite.

\begin{defn} Given a $G$-module $N$, the \emph{associated vector bundle} $N_\mathcal{M}$ on $X$ is: 
\begin{equation}
N_\mathcal{M} := \mathcal{M} \times N / \left\{ (m \cdot g, v) \sim (m, g \cdot v) \right\}.
  \label{eq:assocbundle}
\end{equation}
\end{defn}

\begin{rem}  By construction, the fiber of $N_\mathcal{M}$ over $x\in X$ is non-canonically isomorphic to $N$, and we have an isomorphism,
\begin{equation}\Gamma(X,N_\mathcal{M})\cong \Gamma(\mathcal{M},\mathcal{M} \times N)^{G}.\label{inv-sections}\end{equation}
\end{rem}

We have an associated $G$-equivariant exact sequence of sheaves on $\mathcal{M}$:
\begin{equation}
  0 \to \mathcal{T}_{\mathcal{M}/X} \to \mathcal{T}_\mathcal{M} \to \rho^* \mathcal{T}_X \to 0
\end{equation}
which, by descent gives us an exact sequence of sheaves on $X$ (known as the \emph{Atiyah extension})
\begin{equation}
  0 \to \g_{\mathcal{M}} \xrightarrow{i} \mathcal{E}_\mathcal{M} \xrightarrow{j} \mathcal{T}_X \to 0.
  \label{eq:atiyah}
\end{equation}
Here $\g_\mathcal{M}$ denotes the bundle associated to the $G$-module $\g$, as explained above.  Recall that a connection on $\mathcal{M}$ is a splitting $\theta_\mathcal{M}:\mathcal{E}_\mathcal{M} \to \g_\mathcal{M}$.  We will need to consider a more general notion of connection, taking values in a larger Lie algebra than $\g$.

\begin{defn} A \emph{Harish-Chandra pair $\GH$} is a pro-algebraic group $G$, together with a Lie algebra $\h$ with $G$-action, and a compatible inclusion $\iota:\g \into \h$.
\end{defn}

\begin{defn} A connection $\theta_\mathcal{M}$ for $\GH$ is a map $\theta_\mathcal{M}:\mathcal{E}_\mathcal{M} \to \h_\mathcal{M}$ such that $\iota=\theta_\mathcal{M}\circ i$.  The pair $(\mathcal{M},\theta_\mathcal{M})$ is called a $\GH$-torsor with connection.
\end{defn}

\begin{defn}
For $G$ connected, we say that an $\h$-module $N$ is a $\GH$-\emph{module} if the induced action of $\g$ is integrated to an algebraic representation of $G$.
\end{defn}

Given a $\GH$-torsor with connection $(\mathcal{M}, \theta_{\mathcal{M}})$ a $\GH$-module $N$, the vector bundle $N_\mathcal{M}$ carries a connection $\nabla$, called the Harish-Chandra connection, given by the formula, for $s\in \Gamma(U,N_\mathcal{M})$ and $\xi\in\Gamma(U,\mathcal{T}_X)$:
$$\nabla_\xi(s) := \widetilde{\xi}\cdot s - \theta_\mathcal{M}(\widetilde{\xi})\cdot s,$$
where $\widetilde{\xi}$ is any lift of $\xi$ to $\mathcal{E}_\mathcal{M}$; in fact $\nabla_\xi$ is independent of the choice of lift.

\begin{rem}
Given a pair $\GH$, one can define a formal group $H$ which contains $G$ and has Lie algebra $\mathfrak{h}$ such that $H/G$ is the formal neighborhood of zero in $\mathfrak{h}/\mathfrak{g}$. Then there is an $H$-torsor associated to any $G$-torsor, and splitting an analogue of (\ref{eq:atiyah}) for this $H$-torsor is equivalent to a connection on a $\GH$-torsor as defined above.
\end{rem}

\subsection{Formal geometry}
\label{sec:FGReview}
The fundamental example of a Harish-Chandra torsor we will use is the \emph{bundle of formal coordinate systems}, whose construction we now recall.

Let $\Dn$ denote the formal disc, $\Dn=spf(\C[[x_1, \dots, x_n]])$; denote the unique maximal ideal by $\mathfrak{m}_0$.  We have the pro-algebraic group $Aut(\Dn) = \varprojlim Aut^k(\Dn)$, the inverse limit over $k$ of the automorphism groups $Aut^k(\Dn)$ of $k$-th order infinitesimal neighborhoods of the origin in $\mathbb{A}^n$.

Let $W_n$ denote the Lie algebra of vector fields on $\Dn$, and let $\hW_n^0$ denote the subalgebra of vector fields vanishing at the closed point and more generally let $W_n^k$ denote vector fields vanishing to order $k+1$ at the closed point.  In other words, $W_n$ is the Lie algebra of derivations of $\C[[x_1,\ldots,x_n]]$, $W_n^0$ is the subalgebra preserving the augmentation ideal, and $W_n^k$ is the subalgebra $\mathfrak{m}_0^{k+1}W_n$, under the natural $\C[[x_1,\ldots, x_n]]$-module structure on $W_n$.  We have an identification $W_n^0=Lie(Aut(\Dn))$.

The following proposition is well known.

\begin{prop}\label{fdimrep}
A finite dimensional $W_n^0$-module $N$ integrates to an $Aut(\Dn)$-module if, and only if, the Euler operator $E=\sum x_i\partial_i$ acts diagonalizably with integer eigenvalues.  Moreover, for such $N$ there exists $k$ such that the $Aut(\Dn)$-action on $N$ factors through $Aut^k(\Dn)$.
\end{prop}

Using the above proposition it is easy to confirm that all $W_n$-modules encountered in this section give rise to $\langle Aut(\Dn), \hW_n \rangle$-modules.
\begin{example}
Examples of pro-finite representations of $Aut(\Dn)$ coming from geometry include power series rings $M = k[[x_1, \dots, x_n]]$, and 
more generally, the completion of a natural vector bundle at a smooth point on a variety of dimension $n$ (see Section \ref{sec-nvb}).
\end{example}

\begin{defn} The \emph{bundle of formal coordinate systems} $\mathcal{M}_X$ on a smooth scheme $X$ of dimension $n$ consists of pairs $(x \in X, t_x: X_{(x)} \stackrel{\sim}{\to} \Dn)$.  Projection to the first factor makes $\mathcal{M}_X$ a $Aut(\Dn)$-torsor over $X$. 
\end{defn}

\begin{rem}This definition is incomplete, in that it doesn't define $\mathcal{M}_X$ as a scheme: in fact, $\mathcal{M}_X$ has the structure of a scheme of infinite type over $X$, and is locally trivial as a $Aut(\Dn)$-torsor in the Zariski topology on $X$.  We refer the reader to \cite{FBZ}, Sections 9.4.4 and 11.3.3, and \cite{BK}, Section 3.1, for complete definitions regarding the scheme structure on $\mathcal{M}_X$.\end{rem}

The action of $\hW_n$ on $\C[[x_1,\dots,x_n]]$ defines a map,
$$\tilde{a}:\hW_n\times\mathcal{M}_X\to T_{\mathcal{M}_X/X},$$
descending to an isomorphism,
$$a: (\hW_n)_\mathcal{M} \xrightarrow{\sim} \mathcal{E}_\mathcal{M}.$$
\begin{defn} The \emph{bundle of flat formal coordinate systems} on $X$ is the $\langle Aut(\Dn), \hW_n \rangle$-torsor, ($\mathcal{M}_X,\theta_{\mathcal{M}_X}=a^{-1}$).
\end{defn}

\begin{rem}\label{connection-one-form}Given a $\langle Aut(\Dn),\hW_n \rangle$-module $N$, and $U\subset X$ over which $N_{\mathcal{M}_X}$ trivializes, the Harish-Chandra connection $N_{\mathcal{M}_X}|_U=N_{\mathcal{M}_U}$ is given by:
$$ \nabla_\xi(s) = ds(\xi) - \xi\circ s.$$
Thus a section $s \in \Gamma(X,N_{\mathcal{M}_X})$ is flat if the derivative, $\xi(s)$, of $s$ along the base is simply the action of $\xi$ in the fiber.\end{rem}

\begin{defn} \label{def:glob} The globalization $\mathcal{GL}_X(N)$ to X of a $\langle Aut(\Dn),\hW_n\rangle$-module $N$ is the sheaf of flat sections of the associated bundle $(N_{\mathcal{M}_X},\nabla)$.
\end{defn}

\subsection{Natural vector bundles}\label{sec-nvb}
A natural vector bundle is essentially a functorial assignment of a vector bundle $\V(X)\to X$, to every smooth $X$.  More precisely, we have:

\begin{defn}
The category $Sm_n$ has as objects smooth, finite-type formal schemes, of dimension $n$, and as morphisms \'etale maps.
\end{defn}

\begin{rem} Thus, the category $Sm_n$ is just general enough to encompass both ordinary smooth algebraic varieties, as well as formal neighborhoods of smooth points.\end{rem} 

\begin{defn}
The category $VB_n$, of vector bundles with $n$-dimensional base, has as objects pairs $(X,V)$ of a smooth, finite-type formal scheme $X$ of dimension $n$, and a vector bundle $V\to X$ over $X$.  A morphism from $(X,V)$ to $(Y,W)$ is an \'etale map $f:X\to Y$, together with an isomorphism $\phi:V\to f^*(W)$.
\end{defn}

We have a forgetful functor $F:VB_n\to Sm_n$, forgetting the vector bundle.

\begin{defn}\label{nvb-def} A natural vector bundle $\V$ on $n$-dimensional schemes is a functorial splitting of $F$, i.e. a functor $G:Sm_n\to VB_n$, such that $G\circ F=\operatorname{id}_{Sm_n}$.  A morphism of natural vector bundles is a natural transformation of the corresponding functors.  This defines the category, $\mathcal{NVB}_n$, of natural vector bundles.
\end{defn}

Semi-simple examples of natural vector bundles include: trivial vector bundles, tangent and cotangent bundles, differential forms, polyvector fields, and more generally any tensor bundle built from tangent and cotangent bundles by Schur functors.  Non-semi-simple examples include bundles of polynomial differential operators of order $\leq k$, for any $k$.

\begin{defn}
Let $i_\mathbf{0}: \{0\} \into \Dn$ denote the inclusion of the origin into $\Dn$.  The \emph{standard fiber} $St(\V)$ of $\V\in\NVB_n$ is $St(\V)=i_\mathbf{0}^*(\V(\Dn))$. 
\end{defn}

Applying Definition \ref{nvb-def}, each automorphism $g:\Dn\to\Dn$ induces an isomorphism $\phi_g:\V(\Dn)\to g^*\V(\Dn)$.  Since $g$ preserves the closed point, the assignment $ g\mapsto \phi_g$ defines an action of $Aut(\Dn)$ on $St(\V)$.  Thus we have the \emph{standard fiber} functor:
$$St: \NVB_n\to Aut(\Dn)\textrm{-mod}_f,$$
to the category of finite dimensional algebraic $Aut(\Dn)$-modules.

Conversely, given a finite dimensional algebraic $Aut(\Dn)$-module $N$, the assignment to each smooth scheme $X$ of its associated bundle $N_\mathcal{M_X}$ determines a natural vector bundle, $Assoc(N)$.  
\begin{prop}
\label{prop:AssocSt}
The functors $St$ and $Assoc$ are mutually inverse equivalences of categories.
\end{prop}
\begin{proof}
Let $M \in Aut(D_n)-\text{mod}_f$. We have a natural isomorphism $M \stackrel{\sim}{\to} St \circ Assoc(M)$, which simply identifies $M$ with the fiber of the associated bundle over the closed point.

For $\V \in \NVB_n$, we now construct a natural isomorphism $\V \to Assoc(St(\V))$.  Set $\mathcal{G} = Assoc \circ St (\V)$. Given $X \in Sm_n$, the data of a map $\Psi: \mathcal{G}(X) \to \V(X)$ is equivalent to that of a $G$-equivariant map $\widetilde{\Psi}$ as shown:
    \[
\xymatrix{
\mathcal{M}_X \ar^\pi[d] & \pi^*\mathcal{G}(X) \ar[r]^{\widetilde{\Psi}}="a" \ar@{-->}[d] & \pi^*\V(X) \ar@{-->}[d] \\
X & \mathcal{G}(X) \ar[r]^\Psi="b" & \V(X)
\ar@{-->}@/^0pc/"a";"b"
}
    \]
with dashed arrows indicating the descent functor (i.e. taking $Aut(\Dn)$-invariants). Note that any point $(x, \varphi) \in \mathcal{M}_X$ gives a map $\varphi_x: St(\V) \stackrel{\sim}{\to} \V_x$, the fiber of $\V$ at $x \in X$. Now let $s$ be a section of $\pi^*\mathcal{G}(X)$ over $\mathcal{M}_X$. We define: 
    \[ \left(\widetilde\Psi(s)\right)(x, \varphi) := (x, \varphi_x(s(x))). \]
It is easy to see that $\widetilde{\Psi}$ is a $G$-equivariant isomorphism, hence it induces an isomorphism $\Psi$.
\end{proof}

Globalizations of certain $W_n$-equivariant $\mathcal{O}_{\Dn}$-modules provide an important source of natural vector bundles.  We outline the construction now.

\begin{defn}
Let $Vec(\Dn/W_n)$ denote the category of finite rank $\C[[x_1,\ldots, x_n]]$-modules with a compatible $\langle Aut(\Dn),W_n\rangle$-module structure. 
\end{defn}

\begin{rem}  Equivalently, and object of $Vec(\Dn/W_n)$ is a finite rank $\C[[x_1,\ldots, x_n]]$-module $M$ with compatible $W_n$-action, such that the Euler operator $E$ acts diagonalizably with integer eigenvalues on the associated graded module, $gr(M)=\oplus_k \mathfrak{m}_0^kM/\mathfrak{m}_0^{k+1}$.\end{rem}
We note that the globalization functor $\GL_X$ applied to the $\hW_n$-module $\C[[x_1,\ldots,x_n]]$ yields jets of the structure sheaf $\mathcal{O}_X$, compatibly with the action of $\mathcal{O}$ on any $M\in Vec(\Dn/W_n)$.  This endows the globalization $\GL_X(M)$ of $M\in\mathcal{W}^{int}_\mathcal{O}\textrm{-mod}$ with the structure of a vector bundle.  The naturality of the torsor $\mathcal{M}_X$ under \'etale morphisms gives a functor:

$$\GL: Vec(\Dn/W_n) \to \mathcal{NVB}_n.$$

Conversely, given a natural vector bundle $\mathcal{V}$, the evaluation $ev_{\Dn}: \mathcal{V}\mapsto\mathcal{V}(\Dn)$ is naturally an object of $Vec(\Dn/W_n)$.

\begin{prop}
The functors $ev_{\Dn}$ and $\GL$ define mutually inverse equivalences of categories.
\end{prop}
\begin{proof}
A natural isomorphism $\GL \circ ev_{\Dn} (\V) \to \V$ can be constructed as in the proof of \ref{prop:AssocSt}.  For $M\in Vec(\Dn/W_n)$-mod, we construct a natural isomorphism $M\to ev_{\Dn}\circ \GL(M)$, as follows.  Let $m \in M$. We have a canonical trivialization of the associated bundle $M_{\mathcal{M}_{\Dn}}$ over $\Dn$.  For $m\in M$, we define a section $f_m \in \Gamma(\Dn, \GL_{\Dn}(M)),$ by: 
\[ f_m(y_1, \dots, y_n) = e^{\sum_iy_i\partial_i} \cdot m = \sum_{I}\frac{1}{I!} y^I(\partial_Im). \]
On the right hand side, $\partial_i$ acts on $m\in M$ via the $W_n$-action, $I$ ranges over all multi-indices $I=(i_1,\ldots i_n)$ with each $i_k\geq 0$, $y^I=y^{i_1}\cdots y^{i_n}$, and $\partial_I = \partial_{i_1}\ldots\partial_{i_n}$.  In other words, $f_m$ is the unique flat section taking value $m$ at the closed point.  The correspondence $m\mapsto f_m$ is clearly injective and surjective, since a flat section on $\Dn$ is uniquely determined by its value at the origin. 

\end{proof}

In summary, we have the following categories, and functors between them:
\begin{equation}\label{eq:nvbun}
\xymatrix{Vec(\Dn/W_n) \ar^{\GL}_\sim@/^/[rr]  && \ar^{ev_{\Dn}}@/^/[ll] \mathcal{NVB}_n \ar^{St}_\sim@/^/[rr]  && \ar^{Assoc}@/^/[ll] Aut(\Dn)\text{-mod}_f 
}
\end{equation}

\begin{rem} \label{fdequiv}By Proposition \ref{fdimrep}, the equivalences outlined above reduce the problem of studying natural vector bundles to that of studying finite dimensional representations of the algebraic groups $Aut^k(\Dn)$.  These groups are unipotent extensions of the group $GL_n$; as such their finite dimensional irreducible representations are all pulled back from $GL_n$.
\end{rem}

\subsection{Lower central series of locally free algebras}
Having reviewed the basic setup for formal geometry and natural vector bundles, we now explain the relation to lower central series.
\begin{lem}
Let $\widehat{A}$ be a local ring, non-canonically isomorphic to $\widehat{A}_n$.  Then any isomorphism $\psi:\widehat{A}/M_2 \stackrel{\sim}{\to} \C[[x_1, \ldots, x_n]]$ induces a canonical isomorphism $\widetilde{\psi}:N_k(\widehat{A}) \stackrel{\sim}{\to} N_k(\widehat{A}_n)$.  Moreover, for $g\in Aut(\Dn)$, we have $\widetilde{g\circ \psi}=\widetilde{g}\circ\widetilde{\psi}$.\label{lem:homlift}
\end{lem}
\begin{proof}  We have the Feigin-Shoikhet natural isomorphisms $\xi:(\widehat{A}/M_3)_\star\cong \Omega^{ev}(\widehat{A}/M_2)$.  Recall that $\Omega^{ev}(-)$ is functorial w.r.t isomorphisms: an isomorphism $f: A_1\xrightarrow{\sim} A_2$ of commutative algebras induces an isomorphism $\Omega^{ev}(f):\Omega^{ev}(A_1)\xrightarrow{\sim}\Omega^{ev}(A_2)$.  Thus, given $\psi$, we have the natural isomorphism $\overline{\Psi}$, defined as the composition: 
  \[
    \xymatrix{
    (\widehat{A}/M_3)_\star \ar^\sim_\xi[r]\ar@{-->}[d]^{\overline{\Psi}}&\Omega^{ev}(\widehat{A}/M_2) \ar[d]^{\Omega^{ev}(\psi)}\\
    (\widehat{A}_n/M_3)_\star& \ar^{\xi^{-1}}_\sim[l] \Omega^{ev}(\widehat{A}_n/M_2)
    }
  \]
Let $\Psi:\widehat{A} \to \widehat{A}_n$ denote an arbitrary isomorphism lifting of $\overline{\Psi}$.  We claim that $\Psi|_{N_k(A)}$ is independent of the choice of lift.  Suppose that $\Psi'$ is another lift; then $\Psi'(a)-\Psi(a)\in M_3(\widehat{A}_n)$, for all $a\in \widehat{A}$.  The containments $M_3M_k\subset M_{k+2}$, and $[A,M_3]\subset L_4$ (Theorem 1.3 from \cite{BJ}) imply that $\Psi$ and $\Psi'$ agree modulo $M_{k+1}(\widehat{A}_n)$:
  \begin{align*}
    \Psi(a_0[a_1, [a_2, \dots, [a_{k-1}, a_k]]\dots) &=\Psi(a_0)[\Psi(a_1), \dots[\Psi(a_{k-1}), \Psi(a_k)]\dots]\\
						     &=\Psi'(a_0)[\Psi'(a_1), \dots[\Psi'(a_{k-1}), \Psi'(a_k)]\dots] \mod M_{k+1}.
\end{align*}
Thus, we may set $\widetilde{\psi}:=\Psi.$  The second part of the claim follows from the naturality in constructing $\overline{\Psi}$.
\end{proof}

Similarly, any $\chi\in \hW_n$ induces an endomorphism of $N_k(\widehat{A_n})$, compatible with the identification $W_n^0\cong Lie(Aut(\Dn))$.  As a consequence, each component $N_k(\widehat{A_n})$ is a module for the pair $(Aut(\Dn),\hW_n)$.  Thus, we have the sheaf $\GLNk$ on $X$.  The main result of this section is the following:
\begin{thm}
\label{thm:formloc}
We have $\nks(A) \cong \GLNk$ as vector bundles on $X$.  In particular, we have $N_k(A) \cong \Gamma(X, \GLNk)$, as $U(Vect(X))\ltimes \mathcal{O}_X$-modules.
\end{thm}
Applying the construction of Proposition \ref{prop:astar-assoc} we may regard $N_k(\widehat{A}_n)$ as an object of $\mathcal{W}^{int}_{n,\mathcal{O}}$-mod.  Thus, combining Theorem \ref{thm:formloc} and the existence of Kapranov's locally free lifts $A(X)$ of any smooth $X$, we have:
\begin{cor}\label{cor:nvb} The assignment $X\mapsto N_k(A(X))$, determines a natural vector bundle.\end{cor}
\begin{rem} It should be noted here (see \cite{K}, Remark 1.6.4) that the thickenings $A(X)$ constructed by Kapranov are not themselves functorial in $X$, only the components $N_k(A(X))$ are, by Theorem \ref{thm:formloc}.
\end{rem}

\begin{proof}[Proof of Theorem \ref{thm:formloc}]
Let $p:\MX\to X$ denote the canonical projection, and let $U\subset X$ open.  
For each $x\in X$, we have canonical maps,  
\begin{equation}
\label{eq:rhox}
  \rho_x: \nks(U) \to (\nks)_{(\mathfrak{m}_x)} \stackrel{\sim}{\to} N_k(A_{(\mathfrak{m}_x + M_2)}),
\end{equation}
composing the restriction to stalks with the isomorphism of Corollary \ref{cor:ncloc}.  We define:
$$\varphi:  \nks(U) \to \Gamma(p\inv(U), \MX \times N_k(\widehat{A}_n)),$$
$$\varphi(f)(x, \psi) = \widetilde{\psi}(\rho_x(f)),$$ 
where $\widetilde{\psi}$ is the lift of $\psi$ constructed in Lemma \ref{lem:homlift}.

\begin{prop}
The section $\varphi(f)$ is $\Ghat$-equivariant, for all $f \in \nks(U)$.
\end{prop}
\begin{proof}
We compute:
\begin{align*}
    g \cdot\left(\varphi(f)(x, \psi)\right) &= g \cdot (\widetilde{\psi}(\rho_x(f))) \\
    &= \widetilde{g}(\widetilde{\psi}(\rho_x(f)))\\
    &= (\widetilde{g \circ \psi})(\rho_x(f))\\
    & = \varphi(f)(x, g\circ\psi).
  \end{align*}
\end{proof}

By descent, $\varphi(f)$ defines a section in $\Gamma(U,N_X)\cong \Gamma(p^{-1}(U),\MX\times N_k(\widehat{A}_n))^{\Ghat}$, which we also call $\varphi(f)$.

\begin{prop}
  \label{prop:flat}
  The section $\varphi(f)$ is flat, for all $f\in \nks(U)$.
\end{prop}
\begin{proof}
A vector field $\xi$ gives rise to a family $\xi_x$ of derivations of $\mathcal{O}(X_{(\mathfrak{m}_x)})$ at each smooth point $x\in X$.  By definition of the Harish-Chandra connection, $\phi(f)$ is flat, and only if, the action of $\xi_x$ on each $N_k(A_{(x)})$ agrees with the action of $\xi$ on $N_k(A)_{(x)}$.  Thus flatness of $\phi(f)$ follows from the compatibility of the isomorphism of Theorem \ref{thm:ncloc} with derivations.
\end{proof}

Thus, the map $\varphi$ defines a map of sheaves, $\varphi:\widetilde{N_k}(A)\to \GLNk$.  It suffices to check that $\varphi$ is an isomorphism on stalks, which follows from Corollary \ref{cor:ncloc}.
\end{proof}

It now follows from Remark \ref{fdequiv} that the data of the lower central series components $N_k$, on any locally free algebra $A$ of rank $n$, are completely determined by the standard fiber of $N_k(A_n)$.  This is computationally useful even for $N_k(A_n)$ itself, as the examples below illustrate.

\begin{example}
A basis for the standard fiber $St(N_3(A_n))=N_3/A_n^+N_3$ is: $$\{[x_i, [x_i, x_j]], [x_i, x_j][x_k, x_l]\}.$$
We conclude that $St(N_3(A_2))$ is the $W_n^0$-module
\[ 0 \to V_{(2,2)} \to St(N_3(A_2)) \to V_{(2,1)} \to 0. \]
The sequence is not split, as $x_2^2\partial_2[x_1,[x_1,x_2]] = 2[x_1,x_2]^2$, modulo $A_n^+N_3$.
This may be contrasted with \cite{EKM}, Proposition 5.3, which derives a \emph{split} short exact sequence:
\[ 0 \to \mathcal{F}_{(2, 2)} \to N_3(A_2) \to \mathcal{F}_{(2, 1)} \to 0 \]
of $W_n$-modules, where $\mathcal{F}_{(m, n)}$ denotes the tensor field module co-induced from the $\mathfrak{gl}_2$-module $V_{(m,n)}$.  However, the splitting does not commute with the action of $\mathcal{O}_{\mathbb{A}^2}$.
\end{example}
\begin{example} In \cite{Ker}, Theorem 1.1, a bound is established on the degree $|\lambda|$  of tensor field modules $\mathcal{F}_\lambda$ appearing in a $W_n$ composition series of $N_k(A_n)$:
\[|\lambda| \leq
\begin{cases}
  2k - 2 & k\text{ odd;} \\
  2k - 2 + 2\lfloor\frac{n - 2}{2}\rfloor & k\text{ even.}
\end{cases}
\]
Kerchev's bound can be re-interpreted as a bound on the degree $|\lambda|$ of the $\gl_n$-representation $V_\lambda$ appearing in $St(N_k(A_n))$, which can be established simply by bounding the degree of elements of a spanning set.  Using \cite{BJ}, Corollary 1.5 it is easy to show that $St(N_k(A))$ is spanned by elements of the form,
$$ a\star[l_1,[\cdots[l_{k-1},l_k]\cdots ],$$
where $l_1$, $l_k$ have degree at most one, $l_2,\ldots, l_{k-1}$ have degree at most 2, and $a$ is a product of simple brackets $[x_i,x_j]$.  If $k$ is odd, then $M_2L_k\subset M_{k+1}$ in that case, proving the bound.  If $k$ is even, $a$ may be non-trivial, but we may assume $a[l_{k-1}]l_k\neq 0 \mod M_3$; otherwise:

\begin{align*}a\star[l_1,[\cdots,[l_{k-1},l_k]\cdots] &= [a\star l_1,[\cdots,[l_{k-1},l_k]\cdots] \\
&= [l_1,[a\star l_2,[\cdots,[l_{k-1},l_k]\cdots]\\
&= \cdots\\
& = [l_1,[\cdots,a\star[l_{k-1},l_k]]\cdots].
\end{align*}

by repeated application of Kerchev's Lemma 2.1.
\end{example}

\bibliographystyle{amsalpha}
\bibliography{LCS}

\newcommand{\etalchar}[1]{$^{#1}$}
\providecommand{\bysame}{\leavevmode\hbox to3em{\hrulefill}\thinspace}
\providecommand{\MR}{\relax\ifhmode\unskip\space\fi MR }
\providecommand{\MRhref}[2]{%
  \href{http://www.ams.org/mathscinet-getitem?mr=#1}{#2}
}
\providecommand{\href}[2]{#2}
\begin{thebibliography}{BEJ{\etalchar{+}}12}

\bibitem[AJ10]{AJ}
N.~Arbesfeld and D.~Jordan, \emph{New results on the lower central series
  quotients of a free associative algebra}, Journal of Algebra \textbf{323}
  (2010), no.~6, 1813--1825.

\bibitem[BB11]{BB}
M.~Balagovi{\'c} and A.~Balasubramanian, \emph{On the lower central series
  quotients of a graded associative algebra}, J. Algebra \textbf{328} (2011),
  287--300. \MR{2745567 (2012b:16110)}

\bibitem[BEJ{\etalchar{+}}12]{BEJKL}
S.~Bhupatiraju, P.~Etingof, D.~Jordan, W.~Kuszmaul, and J.~Li, \emph{Lower
  central series of a free associative algebra over the integers and finite
  fields}, J. Algebra \textbf{372} (2012), 251--274. \MR{2990011}

\bibitem[BJ13a]{BJ}
A.~Bapat and D.~Jordan, \emph{Lower central series of free algebras in
  symmetric tensor categories}, J. Algebra \textbf{373} (2013), 299--311.
  \MR{2995028}

\bibitem[BJ13b]{BoJ}
B.~Bond and D.~Jordan, \emph{The lower central series of the symplectic
  quotient of a free associative algebra}, J. Pure Appl. Algebra \textbf{217}
  (2013), no.~4, 689--699. \MR{2983843}

\bibitem[BK04]{BK}
R.~Bezrukavnikov and D.~Kaledin, \emph{Fedosov quantization in algebraic
  context}, Moscow Mathematical Journal \textbf{4} (2004), no.~3, 559--592.

\bibitem[CF13]{CF}
C.~Cordwell and T.~Fei, \emph{Lower central series quotients of finitely
  generated algebras over the integers, in preparation}.

\bibitem[CQ95]{CQ}
Joachim Cuntz and Daniel Quillen, \emph{Algebra extensions and nonsingularity},
  J. Amer. Math. Soc. \textbf{8} (1995), no.~2, 251--289. \MR{1303029
  (96c:19002)}

\bibitem[DE08]{DE}
G.~Dobrovolska and P.~Etingof, \emph{An upper bound for the lower central
  series quotients of a free associative algebra}, Int. Math. Res. Not. IMRN
  (2008), no.~12, Art. ID rnn039, 10. \MR{2426755 (2009g:16030)}

\bibitem[EKM09]{EKM}
P.~Etingof, J.~Kim, and X.~Ma, \emph{On universal lie nilpotent associative
  algebras}, Journal of Algebra \textbf{321} (2009), no.~2, 697--703.

\bibitem[FBZ01]{FBZ}
E.~Frenkel and D.~Ben-Zvi, \emph{Vertex algebras and algebraic curves},
  Mathematical Surveys and Monographs \textbf{88} (2001).

\bibitem[FS07]{fs}
B.~Feigin and B.~Shoikhet, \emph{On {$[A,A]/[A,[A,A]]$} and on a {$W_n$}-action
  on the consecutive commutators of free associative algebras}, Math. Res.
  Lett. \textbf{14} (2007), no.~5, 781--795. \MR{2350124 (2009b:16055)}

\bibitem[Jen47]{Jen}
S.~Jennings, \emph{On rings whose associated lie rings are nilpotent}, Bull.
  Amer. Math. Soc. \textbf{53} (1947), 593–597.

\bibitem[Kap98]{K}
M.~Kapranov, \emph{Noncommutative geometry based on commutator expansions}, J.
  Reine Angew. Math. \textbf{505} (1998), 73--118. \MR{1662244 (2000b:14003)}

\bibitem[Ker13]{Ker}
G.~Kerchev, \emph{On the filtration of a free algebra by its associative lower
  central series}, J. Algebra \textbf{375} (2013), 322--327. \MR{2998959}

\end{thebibliography}


\begin{thebibliography}{1}

\bibitem{bj}
{\sc Bapat, A., and Jordan, D.}
\newblock Lower central series of free algebras in symmetric tensor categories.
\newblock {\em J. Algebra 373\/} (2013), 299--311.

\bibitem{bkquant}
{\sc Bezrukavnikov, R., and Kaledin, D.}
\newblock Fedosov quantization in algebraic context.
\newblock {\em Arxiv preprint math/0309290\/} (2003).

\bibitem{de}
{\sc Dobrovolska, G., and Etingof, P.}
\newblock An upper bound for the lower central series quotients of a free
  associative algebra.
\newblock {\em Int. Math. Res. Not. IMRN}, 12 (2008), Art. ID rnn039, 10.

\bibitem{FS}
{\sc Feigin, B., and Shoikhet, B.}
\newblock On {$[A,A]/[A,[A,A]]$} and on a {$W_n$}-action on the consecutive
  commutators of free associative algebras.
\newblock {\em Math. Res. Lett. 14}, 5 (2007), 781--795.

\bibitem{K}
{\sc Kapranov, M.}
\newblock Noncommutative geometry based on commutator expansions.
\newblock {\em J. Reine Angew. Math. 505\/} (1998), 73--118.

\end{thebibliography}

\end{document}